\documentclass[reqno,11pt]{amsart}

\usepackage{amsmath,amssymb,amsthm,tikz}
\usepackage{xcolor}

\def \R {\mathbb{R}}

\def \B {\mathcal{B}}

\def \T {\text{tr}}
\def \M {\text{Sym}}
 
\newtheorem{theorem}{Theorem} 
\newtheorem{lemma}{Lemma}
\newtheorem{proposition}{Proposition}  
\newtheorem{corollary}{Corollary}
\newtheorem{definition}{Definition}
\newtheorem{remark}{Remark}

\numberwithin{equation}{section}

\title[Sharp regularity for a singular free boundary problem]{Sharp regularity for a singular fully nonlinear parabolic free boundary problem}

\author[D. J. Araújo]{Damião J. Araújo}
\address{Department of Mathematics, Universidade Federal da Para\'iba, 58059-900, Jo\~ao Pessoa-PB, Brazil}{}
\email{araujo@mat.ufpb.br}

\author[G.S. Sá]{Ginaldo S. Sá}
\address{Department of Mathematics, Universidade Federal da Para\'iba, 58059-900, Jo\~ao Pessoa-PB, Brazil}{}
\email{ginaldo.sa@academico.ufpb.br}

\author[J.M.~Urbano]{Jos\'{e} Miguel Urbano}
\address{King Abdullah University of Science and Technology (KAUST), Computer, Electrical and Mathematical Sciences and Engineering Division (CEMSE), Thuwal 23955-6900, Saudi Arabia and University of Coimbra, CMUC, Department of Mathematics, 3001-501 Coimbra, Portugal}{} 
\email{miguel.urbano@kaust.edu.sa}

\begin{document}


%
%
%


\begin{abstract} 
This paper establishes sharp local regularity estimates for viscosity solutions of fully nonlinear parabolic free boundary problems with singular absorption terms. The main difficulties are due to the blow-up of the source along the free boundary and the lack of a variational structure. The proof combines the power of the Ishii-Lions method with intrinsically parabolic oscillation estimates. The results are new, even for second-order linear operators in nondivergence form. 
\end{abstract}   

\date{\today}

\maketitle

\section{Introduction} \label{intro}

In many applied problems, for example, in models of chemical heterogeneous catalysts or combustion, one wants to understand the behaviour of a quantity, be it the density of a gas-phase reactant or the temperature of a flame, as it transitions from being positive to suddenly vanishing. This type of phenomenon is accounted for in the pertaining partial differential equations through a source term of the form $u^{\gamma -1}$. The quenching problem (aka Alt-Phillips, see \cite{AP}) corresponds to the case $0<\gamma <1$ and can be seen as an intermediate scenario between the obstacle problem, which is the case $\gamma=1$, and the cavity Bernoulli problem (aka Alt-Caffarelli, see \cite{AC}), which corresponds to $\gamma=0$. In all these problems, the set $\Gamma:= \partial \{u >0 \}$ is a free boundary, and the primary analytical challenges are understanding the behaviour of $u$ near $\Gamma$ and its geometric properties. Investigating the regularity of  $\Gamma$ often requires obtaining sharp regularity results for $u$.

In this paper, we do precisely that for a class of nonvariational parabolic partial differential equations with singular absorption terms of the form
\begin{equation}\label{E0}
	F(x,t,D^2u)  -\partial_t u \sim \gamma u^{\gamma-1} \chi_{\{u>0\}}, 
\end{equation}
for $\gamma \in (0,1)$, corresponding to the free boundary problem
$$\left\{
\begin{array}{ccl}
F(x,t,D^2u)  -\partial_t u = \gamma u^{\gamma-1} & \mbox{in } & \{u>0\}\\
u=\left| \nabla u\right| =0 & \mbox{on } & \{u=0\}.
\end{array}
\right.
$$
Here, $F(x, t, M)$ is a fully nonlinear uniformly elliptic operator, satisfying appropriate additional assumptions. Our main result is that limiting solutions of \eqref{E0}, \textit{i.e.}, those obtained as uniform limits of solutions of a regularised problem, are locally of class $C^{1+\beta, \frac{1+\beta}{2}}$, for any
\begin{equation}
\beta < \min \left\{ \frac{\gamma}{2-\gamma},\alpha_F \right\},
\label{gp}
\end{equation}
where $\alpha_F$ is the optimal exponent of the $C^{1+\mu, \, \frac{1+\mu}{2}}-$regularity theory for solutions of $F$-caloric functions, \textit{i.e.}, solutions of $F(x,t,D^2h)-\partial_t h=0$ (see, for example, \cite{daST,W2}). A particular scenario for which the result is new is when $F$ is a linear $(\lambda,\Lambda)-$elliptic operator in nondivergence form, in which case $\alpha_F~=~1$ in \eqref{gp}. The result is a consequence of studying a singularly perturbed problem and obtaining uniform in $\varepsilon$ regularity estimates. These follow from combining sharp decay estimates close to the free boundary with classical regularity estimates for equations with a bounded source away from it. In our analysis, we must overcome the significant hardship of dealing with an inhomogeneity blowing up along the free boundary and the lack of a variational structure in the parabolic setting. For radial symmetry results for solutions to the problem, see \cite{EJS}. 

The history of the problem is long and quite rich. The variational elliptic theory is well understood (see \cite{AP, P1, P2, We}) and connected, in the Laplacian case $\Delta u=\gamma u^{\gamma-1}$, to the minimization of the non-differentiable functional
$$
\int \frac12 |Du(x)|^2+u(x)^\gamma \, dx.
$$
The non-variational elliptic counterpart of our problem was studied in \cite{AT}, where the authors obtain optimal regularity estimates by studying the fine oscillation decay for limiting solutions at points close to the free boundary. Still in the non-variational setting, the obstacle problem was studied in \cite{CPS} and \cite{FSp}, and the cavity problem was treated in \cite{RTU} (see also \cite{C1,CLW1,CLW2,CV} for solutions in the weak sense).

The paper is organised as follows. In the next section, we fix some notation and basic definitions, make explicit our assumptions, and introduce the singularly penalised auxiliary problem. In section \ref{manaira}, we comment on the existence of viscosity solutions for the Dirichlet problem involving the regularised equation. Section \ref{reg_aux_sec} treats an auxiliary singular equation, and the results therein have an independent interest due to the additional gradient dependence. We explore a parabolic version of Jensen–Ishii’s lemma, obtaining sharp H\"older regularity estimates in space, and then use a comparison principle for certain Pucci-type equations with quadratic gradient terms to derive H\"older regularity estimates in time. In section \ref{loc_reg_sec}, we obtain pointwise oscillation estimates in space and time for positive viscosity solutions $u_\varepsilon$ of the regularised equation, exploring the fact that  $u_\varepsilon^{\frac{2-\gamma}{2}}$ solves an equation of the form studied in the previous section. The final two sections contain the proof of the main result, first for the regularised problem and, upon passage to the limit, also for the original free boundary problem.

\section{Notation, assumptions and definitions}
 
Consider a bounded domain $\Omega \subset \mathbb{R}^n$, with a smooth boundary, $T>0$, and let $\Omega_T:=\Omega \times (-T,0]$. Denote with $\partial_p\Omega_T=\partial\Omega\times(-T,0) \cup(\Omega\times\{0\})$ the parabolic boundary of $\Omega_T$. Given $(x_0,t_0) \in \Omega_T$ and $\rho>0$, we define the intrinsic parabolic cylinders
\begin{equation}\label{cyl}
	G_\rho(x_0,t_0) := B_\rho (x_0) \times (t_0-\rho^2, t_0],
\end{equation}
where $B_\rho(x_0)$ is the euclidean ball with centre at $x_0$ and radius $\rho$.
 
 We denote the space of real $n\times n$ symmetric matrices by $\M(n)$. For parameters $\Lambda \geq \lambda > 0$, we consider the extremal Pucci operators 
$$
\mathcal{M}^{-}_{\lambda,\Lambda}(M) =\lambda\sum_{e_i>0}e_i + \Lambda\sum_{e_i<0}e_i\quad\text{and}\quad \mathcal{M}^{+}_{\lambda,\Lambda}(M) =\Lambda\sum_{e_i>0}e_i + \lambda\sum_{e_i<0}e_i,
$$
where  $e_i=e_i(M)$ are the eigenvalues of $M \in \M(n)$. It is easy to see that  
$$
\mathcal{M}^{-}_{\lambda,\Lambda}(M)= \inf_{A \in \mathcal{A}_{\lambda, \Lambda}} \T (AM) \quad \mbox{and} \quad \mathcal{M}^{+}_{\lambda,\Lambda}(M)= \sup_{A \in \mathcal{A}_{\lambda, \Lambda}} \T (AM),
$$
where $\mathcal{A}_{\lambda, \Lambda}:= \{A \in Sym(n) \, | \, \lambda I \leq A \leq \Lambda I\}$. 
 
 \medskip
 
We now set the conditions that shall be assumed for the fully nonlinear second order operator $F:\Omega _T \times \M(n) \to \mathbb{R}$. 

\begin{description}

\item[A1] $F$ is $(\lambda, \Lambda)$-parabolic, \textit{i.e.}, 
\begin{equation}\label{UP}
\mathcal{M}^{-}_{\lambda,\Lambda}(N) \leq F(x,t,M+N)-F(x,t,M) \leq \mathcal{M}_{\lambda,\Lambda}^{+}(N),
\end{equation}
for every $M, N \in \M(n)$ and $(x,t) \in \Omega _T$.

\medskip

\item[A2] $F$ is continuous in space, and there exists a nondecreasing continuous function $\omega:[\,0, \infty) \to [\,0, \infty)$, with $\omega(0)=0$,  called a \textit{modulus of continuity}, such that
\begin{equation}\label{continuityF}
|F(x,t,M)-F(y,t,M)| \leq \omega(|x-y|)\|M\|,
\end{equation}
for every $M \in \M(n)$, $x,y \in \Omega$ and $t \in (-T,0]$.

\medskip

\item[A3] $F$ is $1$-homogeneous, 
\begin{equation}\label{homogeneous}
F(x,t,\tau M) = \tau F(x,t,M),
\end{equation}
for every $\tau \geq 0$, $(x,t) \in \Omega_T$ and $M \in \M(n)$.

\medskip

\item[A4] $F(x,t,0)=0$, for every $(x,t) \in \Omega_T$.

\end{description}

\smallskip

Following \cite{IS, W1}, we next define the notion of viscosity solution for a fully nonlinear parabolic equation of the form
\begin{equation}\label{def}
F(x,t,D^2u)-\partial_tu=g(x,t,u)\quad\text{in}\ \Omega_T,
\end{equation}
where $g$ is a continuous function in $\Omega _T \times \mathbb{R}$. 
 We say a function belongs to $C^{2,1}(\Omega_T)$ if it is of class $C^2$ in $\Omega$ and of class $C^1$ in $(-T,0]$. 

\begin{definition}
Assume $u \in C(\Omega_T)$. We say $u$ is a viscosity subsolution (supersolution) of \eqref{def} if, for each $\varphi\in  C^{2,1}(\Omega_T)$ such that $u-\varphi$ has a local maximum (minimum) at $(x_0,t_0)\in \Omega_T$, there holds
$$F(x_0,t_0,D^2\varphi(x_0,t_0))-\partial_t\varphi(x_0,t_0)  \geq  (\leq) \ g(x_0,t_0,u(x_0,t_0)). $$
We say $u$ is a viscosity solution if it is both a viscosity subsolution and a viscosity supersolution.
\end{definition}

Next, we focus on the singular penalisation strategy adopted in this paper. Let $\varrho \in C^{\infty}(\mathbb{R})$ be compactly supported in $[0,1]$, with $\int \varrho (\theta) d\theta = 1$. For parameters $0<\sigma_0<1$ and 
\begin{equation}
\alpha=\frac{\gamma}{2-\gamma}
\label{alpha}
\end{equation} 
(recall $\gamma \in (0,1)$), define, for each $\varepsilon>0$, the real function
\begin{equation}
	\mathcal{B}_\varepsilon(s)=\gamma\int_0^{\frac{s-\sigma_0\varepsilon^{1+\alpha}}{\varepsilon^{1+\alpha}}}\varrho(\theta)d\theta, \quad s\in \R, 
\end{equation}
which is an approximation of the characteristic function $\chi_{\{s>0\}}$. We consider the penalised equation
\begin{equation}\label{Ee}
	\tag{$E_\varepsilon$}
	F(x,t,D^2u)  -\partial_t u = \mathcal{B}_\varepsilon(u)\,u^{\gamma-1} \quad \mbox{in } \; \Omega_T. 
\end{equation}

\begin{remark}\label{rem}
We stress the scaling invariance of solutions for \eqref{Ee}. If $v$ is a viscosity solution of \eqref{Ee} in $G_1 \Subset \Omega_T$, then, for any  parameters $\kappa >0$ and $\theta \geq 0$, the rescaled function
$$
v_\kappa(x,t)=\frac{v(\kappa x, \kappa^2t)}{\kappa^\theta} 
$$
solves 
$$
F_\kappa(x,t,D^2v_\kappa) -\partial_t v_\kappa = \kappa^{\theta (\gamma -1) +2 - \theta} \, \mathcal{B}_{\epsilon}(v_\kappa)\,{v_\kappa}^{\gamma-1},
$$
in the viscosity sense in $G_{1/\kappa}$, where 
$$\epsilon =\frac{\varepsilon}{\kappa^{\frac{\theta}{1+\alpha}}} \qquad {\rm and} \qquad F_\kappa(x,t,M):=\kappa^{2-\theta}F(\kappa x,\kappa^2 t,\kappa^{\theta-2}M), $$
which is $(\lambda, \Lambda)$-parabolic  as defined in \eqref{UP}.  
\end{remark}

Finally, we give the notion of regularity in space and time that we shall assume in this paper. For a given $\mu \in (0,1)$, we say a function $u: \Omega_T \to \R$ is of class $C^{\mu,\frac{\mu}{2}}$ at a point $(y,s) \in \Omega_T$ if
\begin{equation}\label{regpoint}
	\sup\limits_{(x,t) \in G_\rho(y,s)} |u(x,t)-u(y,s)| \leq C \rho^{\mu},
\end{equation}
for every $0<\rho \ll 1$ such that $G_\rho(y,s) \Subset \Omega_T$. Additionally, for the borderline case $\mu=1$, we use the notation $C^{Lip, \frac{1}{2}}$ to mean Lipschitz regularity in space and $\frac{1}{2}-$H\"older regularity in time. 

We say $u$ is of class $C^{1+\mu,\frac{1+\mu}{2}}$ at a point $(y,s) \in \Omega_T$ if
\begin{equation}\label{regpoint2}
	\sup\limits_{(x,t) \in G_\rho(y,s)} |u(x,t)-u(y,s)-\nabla u(y,s)\cdot (x-y)| \leq C \rho^{1+\mu},
\end{equation}
for every $0<\rho \ll 1$ such that $G_\rho(y,s) \Subset \Omega_T$. 

Such choices are natural in the face of the intrinsic geometry assumed in \eqref{cyl}, which is suitable for dealing with the homogeneity of equation \eqref{E0}.

\section{Existence of viscosity solutions} \label{manaira}

In this section, we discuss the existence of viscosity solutions for  the Dirichlet problem
\begin{equation}\label{Pe}
	\left\{
	\begin{array}{rclcl}
		F(x,t,D^2u) - \partial_t{u} & = & \mathcal{B}_\varepsilon(u)u^{\gamma-1} &\text{in}& \Omega_T,\\[0.15cm]
		u & = & \varphi &\text{on}& \partial_p \Omega_T,
	\end{array}
	\right.
\end{equation}
with given nonnegative data $\varphi \in C(\partial_p \Omega_T)$. We consider an adaptation of Perron's method as argued in \cite{RTU}. 

Choose functions $u^\star$ and $u_\star = u_\star(\varepsilon)$ solving in the viscosity sense, respectively
$$\left\{
	\begin{array}{rclcl}
		F(x,t,D^2u^\star) - \partial_t{u^\star}  & = & 0 &\text{in}& \Omega_T,\\[0.15cm]
		u^\star & = & \varphi &\text{on}& \partial_p \Omega_T
	\end{array}
\right.
$$
and
$$\left\{
	\begin{array}{rclcl}
		F(x,t,D^2u_\star) - \partial_t{u_\star} & = & \gamma \sigma_0^{\gamma-1}\varepsilon^{(1+\alpha)(\gamma-1)} &\text{in}& \Omega_T,\\[0.15cm]
		u_\star & = & \varphi &\text{on}& \partial_p \Omega_T.
	\end{array}
\right.
$$
We guarantee the existence of $u^\star$ and $u_\star$ by applying the standard Perron method in the theory of viscosity solutions. In addition, noting that
$$
0 \leq \mathcal{B}_\varepsilon(u)\,u^{\gamma-1} \leq \gamma \sigma_0^{\gamma-1}\varepsilon^{(1+\alpha)(\gamma-1)},
$$
we find that $u^\star$ and $u_\star$ are a viscosity supersolution and a viscosity subsolution of  \eqref{Pe}.
Therefore, according to \cite[theorem 3.1]{RTU}, we conclude that the function
$$
u_\varepsilon:= \inf\limits_{w \in \mathcal{S}}w,
$$
where
$$
\mathcal{S} = \{w\in C(\overline{\Omega_T}) \; | \; u_\star \leq w \leq u^\star \; \text{and} \; w \; \text{is a supersolution to \eqref{Pe}}\},
$$
solves \eqref{Pe} in the viscosity sense.

\begin{remark}\label{ABP}
We observe that $u_\varepsilon$ is uniformly globally bounded. In fact, by using the parabolic Alexandrov–Bakelman–Pucci (ABP) estimate (see, for instance, \cite[theorem 3.2]{RTU}, together with \cite[theorem 3.14]{W1}), we can find a positive constant $C$, depending only on $n$, $\lambda$, $\Lambda$ and $\|\varphi\|_\infty$, but independent of $\varepsilon$, such that
$$
0 \leq u_\varepsilon \leq C \quad \mbox{in } \; \Omega_T.
$$ 
\end{remark}

\begin{remark}	
Solutions $u_\varepsilon$ to \eqref{Pe} are positive in $\Omega_T$, provided $\varphi>0$ in $\partial_p\Omega$. In fact, suppose we can find $-T<t^\star \leq 0$, such that 
$$t^\star = \sup \left\{ t \in (-T,0]\; | \; (x,t) \in \partial\{u_\varepsilon>0\} \cap \Omega_T \right\}.$$
Then $u_\varepsilon$  solves $F(x,t,D^2 u_\varepsilon)-\partial_t u_\varepsilon =0$ in the open set 
$$\mathcal{O}:=\left\{ (x,t) \in \overline{\Omega_T}\, | \, u_\varepsilon (x,t) < \sigma_0 \varepsilon^{1+\alpha} \right\},$$
which is nonempty. Therefore, as a consequence of the parabolic strong maximum principle (see, for instance, \cite{L, W1}), we obtain 
$$
\Omega \times (-T,t^\star] \subset \{u_\varepsilon =0\},$$
which is in contradiction with the (global) continuity of $u_\varepsilon$. 

Given a fixed boundary data $\varphi \geq 0$, we thus only have to take $\varphi_\varepsilon= \varphi+\varepsilon^{1+\alpha}$ in \eqref{Pe}, to obtain positive solutions for \eqref{Pe}.

\end{remark}

Our primary goal in this paper is to derive geometric and analytic properties for positive solutions of \eqref{Ee}, which are uniform in $\varepsilon$. This will lead, by letting $\varepsilon \searrow 0$, to regularity estimates across the free boundary for limiting solutions of \eqref{E0} (see Section \ref{limitingsec}). 

\section{Local H\"older estimates for an auxiliary singular equation}\label{reg_aux_sec}     

We will work in $G_1=B_1 \times [- 1,0]$ throughout the section. We will derive regularity estimates, both in space and time, for positive bounded viscosity solutions of the auxiliary singular equation 
\begin{equation}\label{general_eq}
	F\left(x,t, D^2 v + \delta \,v^{-1} \nabla v \otimes \nabla v \right) - \partial_t v =f(x,t)\,v^{-1},
\end{equation}
for a given parameter $\delta \geq 0$ and $f \in L^\infty (G_1)$. We shall assume the fully nonlinear operator $F$ to satisfy conditions \eqref{UP} and \eqref{continuityF}. 

We now define the parabolic super-/sub-differentials of a function $v$ at the point $(x,t)$,
\begin{eqnarray*}
\mathcal{P}^\pm(v)(x,t) & = & \Big\{ (\alpha, p , X) \in \mathbb{R} \times \mathbb{R}^n \times Sym(n) \ |  \\
&  & \quad  (\alpha, p , X) = \left( \phi_t  (x,t) , \nabla \phi (x,t), D^2 \phi (x,t) \right), \\
& & \qquad \mbox{for $\phi \in C^{2,1}$ touching $v$ from above (below) at $(x,t)$} \Big\},
\end{eqnarray*}
and the corresponding limiting super-/sub-differentials,
\begin{eqnarray*}
\overline{\mathcal{P}}^\pm(v)(x,t) & = & \Big\{ (\alpha, p , X) \in \mathbb{R} \times \mathbb{R}^n \times Sym(n) \ |  \ \exists (x_m, t_m) \to (x,t), \\
&  & \quad \exists  (\alpha_m, p_m , X_m)  \in \mathcal{P}^\pm(v)(x_m,t_m)  \ \mbox{ such that}\\
& & \qquad   (\alpha_m, p_m , X_m) \to  (\alpha, p , X) , v(x_m, t_m) \to v(x,t) \Big\}.
\end{eqnarray*}

\subsection{H\"older regularity estimates in space} We start by obtaining local spatial $C^{0,1^-}-$estimates for positive solutions of \eqref{general_eq}. 

The following is a generalised Jensen–Ishii’s lemma (see \cite{CIL}). 

\begin{lemma}\label{IS_lemma}
	Let $v \in C(G_1)$. Set
	$$
	\varphi(x,y,t):=v(x,t)-v(y,t)-L|x-y|^\mu-K(|x|^2+(-t)^2),
	$$
	for $L$,$K$ positive constants and $\mu \in (0,1)$. If the function $\varphi$ attains a maximum at $(x_m,y_m,t_m) \in B_{1/2} \times B_{1/2} \times (-1/4,0\,]$, then there exist $\tau \in \mathbb{R}$, $p \in \mathbb{R}^n$, and $X,Y \in Sym(n)$, with  $X \leq Y$, such that
	\begin{equation}\nonumber
		\begin{array}{rcl}
			(\tau+2Kt_m,p+2Kx_m, X+2KI) & \in & \overline{\mathcal{P}}^{+}(v)(x_m,t_m), \\[0.15cm] 
			(\tau,p, Y) & \in & \overline{\mathcal{P}}^{-}(v)(y_m,t_m),
		\end{array}
	\end{equation}
	and the following estimates
	\begin{equation}\label{IS_estim}
		|p|=\mu L|x_m-y_m|^{\mu-1}, \quad \|Y\| \leq 2\mu L|x_m-y_m|^{\mu-2} 
	\end{equation}
	and
	\begin{equation}\label{IS_estim2}
		\T(Y-X) \geq 8\mu \frac{1-\mu}{3-\mu}L|x_m-y_m|^{\mu-2}
	\end{equation}
	hold.
\end{lemma} 

\begin{proof}
	We use \cite[Lemma 2.3.30]{IS} to guarantee the existence of $\tau \in \mathbb{R}$, $p \in \mathbb{R}^n$ and $X,Y \in Sym(n)$. Estimates \eqref{IS_estim} and \eqref{IS_estim2} follow by computations in the proof of \cite[Theorem 2.3.29]{IS}, namely from those in lines 2, 8 and -8 on page 41. 
	
\end{proof}

\begin{theorem}\label{thm_holder}
	Let $v$ be a positive bounded viscosity solution of \eqref{general_eq} in $G_1$. For any $\mu \in (0,1)$, there exists a constant $C>1$, depending only on $\mu$, $\delta$, $\lambda$, $\Lambda$, $n$, $\omega (\cdot)$ and $\|v\|_{L^\infty(G_1)}$, such that
	\begin{equation}\label{muestimates}
		\sup\limits_{x,y \in B_{1/2}} \frac{|v(x,t)-v(y,t)|}{|x-y|^\mu} \leq C,
	\end{equation}
	for  every $ \, t \in (-1/4, 0]$.
\end{theorem}

\begin{proof}
Consider the auxiliary function defined in $G_1$ by 
$$\Phi(x,y,t) = v(x,t)-v(y,t)-L\,|x-y|^\mu-K \left(|x|^2+(-t)^2 \right) ,$$
where the parameter $\mu \in (0,1)$, and the constants $L$ and $K$ are to be chosen. We claim that
$$\Phi(x,y,t)\leq 0 \quad \mbox{in } \  \overline{B_{1/2}}\times\overline{B_{1/2}} \times \left[-{1/4},0\right] , $$ 
from which the result follows. Indeed, from $\Phi(x,0,0)\leq 0$ and $\Phi(0,y,0)\leq 0$, we obtain
$$\sup\limits_{x \in B_{1/2}} \frac{|v(x,0)-v(0,0)|}{|x|^\mu} \leq C,$$	
and \eqref{muestimates} follows by translation.

Assuming the claim is false with the goal of reaching a contradiction, suppose $\Phi$ has a positive maximum at some point $(x_m,y_m,t_m)\in  \overline{B_{1/2}}\times\overline{B_{1/2}} \times \left[-{1/4},0\right]$. Then necessarily $x_m\neq y_m$ and we have
$$L\,|x_m-y_m|^\mu+K \left(|x_m|^2+(-t_m)^2 \right) \leq 2\|v\|_\infty.$$
If we choose $K=2^9\|v\|_\infty$ and $L \geq K$, we conclude
$$(x_m,y_m,t_m) \in B_{1/4} \times B_{1/4} \times (-1/16,0].$$
Recalling assumption {\bf (A2)}, we shall assume $L$ even larger and such that
\begin{equation}\label{osc_omega} 
	\omega(|x_m-y_m|) \leq  \min \left\{\frac{\lambda}{2}, \frac{4\lambda}{2+\delta} \frac{1-\mu}{3-\mu}\right\}.
\end{equation}

In addition, observe that  we also have
\begin{equation}\label{contrad}
	v(x_m,t_m)>v(y_m,t_m) \quad \mbox{and} \quad  v(x_m,t_m) > L|x_m-y_m|^\mu.
\end{equation}
	
We now apply Lemma \ref{IS_lemma} considering the previous remarks. Define the quantities
	\begin{equation}\nonumber 
		\begin{array}{rcl}
			A_x &:= & X + 2KI+\delta\,v(x_m,t_m)^{-1}(p+2Kx_m)\otimes (p+2Kx_m), \\[0.2cm]
			A_y &:= & Y+\delta\,v(y_m,t_m)^{-1} p \otimes p, \\[0.2cm]
			F[A_x] &:= & F(x_m,t_m,A_x), \\[0.2cm]
			F[A_y] &:= & F(y_m,t_m,A_y), \\[0.2cm]
			\Delta & := & |x_m-y_m|.
		\end{array}
	\end{equation}
	From \eqref{general_eq}, we have 
	\begin{equation}\nonumber
		\begin{array}{rcl}
			F[A_x] & \geq &  \tau + 2Kt_m +f(x_m,t_m)v(x_m,t_m)^{-1},
			\\[0.15cm]
			F[A_y] & \leq &  \tau +f(y_m,t_m)v(y_m,t_m)^{-1}
		\end{array}
	\end{equation}
	and, consequently,
	\begin{equation}\label{PDE}
		F[A_y]-F[A_x] \leq K+ f(y_m,t_m)v(y_m,t_m)^{-1} - f(x_m,t_m)v(x_m,t_m)^{-1}.
	\end{equation}
	
	On the other hand, by \eqref{continuityF},
	\begin{equation}\nonumber
		F[A_y]-F[A_x] = F(y_m,t_m,A_y) - F(x_m,t_m,A_y)  \geq  -\omega(\Delta)\|A_y\|.
	\end{equation}
	From this fact and \eqref{UP}, for each $\iota>0$, we can find $M_\iota \in  \mathcal{A}_{\lambda, \Lambda}$, such that
	\begin{equation}\label{operator}
		F[A_y]-F[A_x] \geq \T(M_\iota(A_y-A_x)) - \iota -\omega(\Delta)\|A_y\|.
	\end{equation}
We next estimate the first term on the right-hand side of \eqref{operator}.	
\begin{equation}\nonumber
		\begin{array}{c}
			\T(M_\iota(A_y-A_x)) \\[0.2cm]
			= -2K \T({M_\iota}) +\T(M_\iota(Y-X)) + \delta\,v(y_m,t_m)^{-1}\T(M_\iota (p \otimes p)) \\[0.2cm]
			- \delta\,v(x_m,t_m)^{-1}\T(M_\iota (p+2Kx_m)\otimes(p+2Kx_m)) \\[0.2cm]
			\geq -2K \Lambda n +\lambda \T(Y-X) + \delta\,v(y_m,t_m)^{-1}\T(M_\iota (p \otimes p)) \\[0.2cm]
			- \delta\,v(x_m,t_m)^{-1}\left(\T(M_\iota \, p\otimes p)+4K\T(M_\iota \, p\otimes x_m )+4K^2\T(M_\iota\, x_m \otimes x_m)\right) 
			\\[0.2cm]
			\geq -2K \Lambda n +\lambda \T(Y-X) + \delta\,v(y_m,t_m)^{-1}\T(M_\iota (p \otimes p)) \\[0.2cm]
			- \delta\,v(x_m,t_m)^{-1}\left(\T(M_\iota \, p\otimes p)+4K\Lambda n|p||x_m|+4K^2\Lambda|x_m|^2
			\right).
		\end{array}
	\end{equation}
Taking \eqref{PDE}, \eqref{operator} and the previous estimate into account, and noting that 
\begin{equation}\nonumber
		\|A_y\| \leq \|Y\|+\delta\,v(y_m,t_m)^{-1} |p|^2,
	\end{equation}
we conclude
\begin{eqnarray}\label{natal22}
&  K(1 + 2 \Lambda n)-\lambda\T(Y-X)  +\omega(\Delta)\|Y\| +\iota \nonumber \\[0.2cm]
& \geq  v(x_m,t_m)^{-1}\left(f(x_m,t_m)-\delta\T(M_\iota \, p\otimes p)-4\delta K\Lambda n|p||x_m|-4\delta K^2\Lambda|x_m|^2\right) \nonumber \\[0.2cm]
& -v(y_m,t_m)^{-1}\left(f(y_m,t_m)-\delta\,\T(M_\iota (p \otimes p))+\delta \omega(\Delta) |p|^2\right). 
\end{eqnarray}

From \eqref{IS_estim} and the choices made in \eqref{osc_omega}, for $L \gg 1$, we obtain
	\begin{equation}\nonumber
		\begin{array}{rcl}
			f(y_m,t_m)-\delta\,\T(M_\iota (p \otimes p))+\delta \omega(\Delta) |p|^2 & \leq & \|f\|_{\infty}-\delta(\lambda- \omega(\Delta))|p|^2 \\[0.2cm] & \leq & \|f\|_{\infty}-\dfrac{\lambda}{2} \delta\mu^2L^2 \\[0.25cm] 
			& < & 0.
		\end{array}
	\end{equation}
But, from \eqref{contrad}, 
$$-v(y_m,t_m)^{-1} < -v(x_m,t_m)^{-1}	$$
so, letting $\iota \to 0$ in \eqref{natal22}, we obtain  
	\begin{equation}\nonumber
		\begin{array}{c}
			K(1 + 2 \Lambda n)-\lambda\T(Y-X) +\omega(\Delta)\|Y\|
			\\[0.2cm]
			\geq  v(x_m,t_m)^{-1}\left(-2\|f\|_\infty-4\delta K\Lambda n|p|-4\delta K^2\Lambda -\delta \omega(\Delta) |p|^2\right).
		\end{array}
	\end{equation}
Using \eqref{IS_estim}, \eqref{IS_estim2} and \eqref{contrad}, we arrive at the estimate
	\begin{equation}\nonumber
		\begin{array}{c}
			K(1 + 2 \Lambda n)-8\mu\lambda \dfrac{1-\mu}{3-\mu}L\Delta^{\mu-2} + 2\mu L \omega(\Delta)\Delta^{\mu-2}
			\\[0.35cm]
			\geq  L^{-1}\Delta^{-\mu}\left(-2\|f\|_\infty-4\delta K\Lambda n\mu L\Delta^{\mu-1}-4\delta K^2\Lambda-\delta \mu^2\,\omega(\Delta)  L^2\Delta^{2\mu-2}\right).
		\end{array}
	\end{equation}
Now, we multiply both sides by $L^{-1}\Delta^{2-\mu}$, to get
	\begin{equation}\nonumber
		\begin{array}{c}
			\dfrac{K(1 + 2 \Lambda n)}{L}\Delta^{2-\mu}-8\mu\lambda \dfrac{1-\mu}{3-\mu} + 2\mu \omega(\Delta)
			\\[0.35cm]
			\geq  \dfrac{-2\|f\|_\infty-4\delta K^2\Lambda}{L^2}\Delta^{2(1-\mu)}
			-\dfrac{4\delta K\Lambda n}{L}\mu \Delta^{1-\mu}-\delta \mu^2\,\omega(\Delta).
		\end{array}
	\end{equation}
	Hence, since $\Delta \leq 1$ and $\mu \in (0,1)$, we arrive at
	\begin{equation}\nonumber
		\begin{array}{c}
			\dfrac{K(1 + 2 \Lambda n)}{L}-8\mu\lambda \dfrac{1-\mu}{3-\mu} + (2+\delta)\mu \,\omega(\Delta)
			\\[0.35cm]
			\geq  -\dfrac{2\|f\|_\infty+4\delta K^2\Lambda}{L^2}
			-\dfrac{4\delta K\Lambda n\mu}{L} .
		\end{array}
	\end{equation}
Finally, from \eqref{osc_omega}, we derive
$$-4\mu\lambda \dfrac{1-\mu}{3-\mu} \geq  -\dfrac{2\|f\|_\infty+2^{20}\delta \Lambda \|v\|_\infty^2+2^{11}\delta \|v\|_\infty\Lambda n\mu+2^9\|v\|_\infty(1 + 2 \Lambda n)}{L},$$
\textit{i.e.},
$$L \leq  (3-\mu)\dfrac{\|f\|_\infty+2^{19}\delta \Lambda \|v\|_\infty^2+2^{10}\delta \|v\|_\infty\Lambda n\mu+2^8\|v\|_\infty(1 + 2 \Lambda n)}{2\mu\lambda (1-\mu)}$$
and we select $L$ universally large to get the desired contradiction.

\end{proof}

\subsection{H\"older regularity estimates in time.}
We now derive local $C^{0,{\frac{1}{2}}^-}$ estimates in time for positive bounded viscosity solutions of \eqref{general_eq}. The strategy relies on combining Theorem \ref{thm_holder} with a control of the oscillation in time, obtained by constructing subsolutions and supersolutions of \eqref{general_eq} in nonsingular parabolic regions. 

\begin{lemma}\label{space_time}
Let $v$ be a nonnegative continuous function in $\overline{G_1}$ such that 
\begin{equation}\label{v>1}
\begin{array}{cccc}
\mathcal{M}_{\lambda, \Lambda}^+(D^2 v) + c_1
|\nabla v|^2 - \partial_t v & \geq & -M   & \mbox{in } \; \{v>1\} \cap G_1, \\[0.2cm]
\mathcal{M}_{\lambda, \Lambda}^-(D^2 v) + c_2
|\nabla v|^2 - \partial_t v & \leq & M & \mbox{in } \; \{v>1\} \cap G_1,
\end{array}
\end{equation}
in the viscosity sense, for nonnegative parameters $c_1$, $c_2$ and $M$. Assume that, for $L > 1$, 
\begin{equation}\label{osc_space}
\sup\limits_{x \in B_1}|v(x,t)-v(0,t)| \leq L,
\end{equation}
for each $t \in (-1/4,0]$. Then, for 
$$\kappa_0:= \dfrac{1}{4}\min\left\{1,\dfrac{2L}{M+4n\Lambda L+16c_1 L^2} \right\},$$
we have
\begin{equation}\label{osc_time}
	\sup\limits_{t \in (-\kappa_0,0]}|v(0,0)-v(0,t)| \leq 8L. 
\end{equation}
\end{lemma}

The proof of Lemma \ref{space_time} relies on a comparison principle for second-order operators as in \eqref{v>1}. 

\begin{proposition} \label{comparison}
Let $v_1$ and $v_2$ be, respectively, a viscosity supersolution and a viscosity subsolution of 
\begin{equation}\label{comp}
\begin{array}{cccc}
\mathcal{M}_{\lambda, \Lambda}^+(D^2 v) + c
|\nabla v|^2 - \partial_t v + M & = & 0   & \mbox{in } \; \Omega_T, \end{array}
\end{equation}
with $c\geq 0$ and $M \in \mathbb{R}$. If $v_2 \leq v_1$ in $\partial_p \Omega_T$, then 
$$v_2 \leq v_1 \quad \mbox{in }\  \Omega_T.$$
The same holds for the equation
\begin{equation}\label{comp2}
\begin{array}{cccc}
\mathcal{M}_{\lambda, \Lambda}^-(D^2 v) + c
|\nabla v|^2 - \partial_t v + M & = & 0   & \mbox{in } \; \Omega_T. \end{array}
\end{equation}   
\end{proposition}

\begin{proof}
We adapt arguments in \cite{IS} and argue by contradiction. Assume that, for $\eta>0$ arbitrary,
$$
M_\eta := \sup\limits_{(x,t) \in \Omega_T} \left\{v_2(x,t)-v_1(x,t) + \frac{\eta}{t}\right\}>0.
$$
Note that then, since $t<0$, $M_\eta$ is attained in $\Omega_T$. 

For each $\epsilon>0$, consider
$$
M_\eta^\epsilon := \sup\limits_{x,y \in \Omega,\, t \in (-T,0)} \left\{v_2(x,t)-v_1(y,t)-\frac{|x-y|^2}{2\epsilon} + \frac{\eta}{t}\right\} \geq M_\eta >0,
$$
and denote by $(x_\epsilon,y_\epsilon,t_\epsilon)$ a maximizer. From \cite[Lemma 2.3.19]{IS}, and taking \cite[Remark 2.3.20]{IS} into account, we conclude 
$$
\frac{|x_\epsilon-y_\epsilon|^2}{\epsilon} \to 0, \quad \mbox{as} \quad \epsilon \to 0.
$$
Consequently, $M_\eta^\epsilon \to M_\eta$, as $\epsilon \to 0$. In addition, for $\epsilon \ll 1$, $(x_\epsilon,y_\epsilon,t_\epsilon)$ is an interior point of $\Omega \times \Omega \times (-T,0)$. Now, we apply \cite[Lemma 2.3.23]{IS} for $v_1$ and $v_2+\eta/t$, obtaining the existence of $\tau \in \mathbb{R}$ and $X,Y \in \M(n)$, satisfying $X \leq Y$,
$$
(\tau+\frac{\eta}{t^2},p_\epsilon,X) \in \overline{\mathcal{P}}^+ v_2(t_\epsilon,x_\epsilon) \quad \mbox{and} \quad (\tau,p_\epsilon,Y) \in \overline{\mathcal{P}}^- v_1(t_\epsilon,y_\epsilon),
$$
with $p_\epsilon=\frac{x_\epsilon-y_\epsilon}{\epsilon}$. From \eqref{comp}, we write the following inequalities
$$
\mathcal{M}_{\lambda, \Lambda}^+(Y) + c
|p_\epsilon|^2 - \tau + M \leq 0 \leq \mathcal{M}_{\lambda, \Lambda}^+(X) + c
|p_\epsilon|^2 - \left(\tau+\frac{\eta}{t^2}\right) + M
$$
and so,
$$
\mathcal{M}_{\lambda, \Lambda}^+(X) \leq \mathcal{M}_{\lambda, \Lambda}^+(Y)  \leq  \mathcal{M}_{\lambda, \Lambda}^+(X) - \frac{\eta}{t^2},
$$
which is a contradiction. Thus $v_2(x,t) \leq v_1(x,t) -\frac{\eta}{t}$ in $\Omega_T$ and, since $\eta>0$ is arbitrary, the result follows. 

Similarly, we obtain the result for equation \eqref{comp2}.

\end{proof}  

We are now ready to prove Lemma \ref{space_time}.

\begin{proof}[Proof of Lemma \ref{space_time}]
First, we claim that if
$$
G_{\tau_1,\tau_2}:= B_1 \times (\tau_1,\tau_2] \subset \{v>1\},
$$
for $(\tau_1,\tau_2] \subset (-\kappa_0,0]$, then 
\begin{equation}\label{claim}
|v(0,\tau_2)- v(0,\tau_1)| \leq 2L.
\end{equation}
Indeed, consider functions
$$
h^\pm(x,t):= v(0,\tau_1)\pm L \pm K |x|^2 \pm \overline{K}(t-\tau_1),
$$
for positive parameters $K$ and $\overline{K}$ to be chosen later. Observe that
\begin{equation}\label{pucci}
\begin{array}{rcl}
\mathcal{M}_{\lambda, \Lambda}^+(D^2h^+) +c_1|\nabla h^+|^2- \partial_t h^+ & \leq & 2n\Lambda K + 4c_1K^2 - \overline{K} \\[0.2cm]
-\mathcal{M}_{\lambda, \Lambda}^-(D^2h^-) -c_2|\nabla h^-|^2+ \partial_t h^- & \leq & 2n\Lambda K - \overline{K}.
\end{array}
\end{equation}
Hence, choosing 
$$
\overline{K} = 2n\Lambda K +4c_1K^2+M > 2n\Lambda K +M,
$$
we obtain
$$\mathcal{M}_{\lambda, \Lambda}^+(D^2 v) + c_1
|\nabla v|^2 - \partial_t v  \geq -M \geq  \mathcal{M}_{\lambda, \Lambda}^+(D^2h^+) +c_1|\nabla h^+|^2- \partial_t h^+  $$
$$\mathcal{M}_{\lambda, \Lambda}^-(D^2 v) + c_2
|\nabla v|^2 - \partial_t v  \leq \,M  \leq  \mathcal{M}_{\lambda, \Lambda}^-(D^2h^-) +c_2|\nabla h^-|^2 - \partial_t h^-.$$
In the sequel, we define 
$$
t^\star:= \sup\limits_{\tau_1 \leq s \leq \tau_2} \left\{ s \; : \; |v(0,t)-v(0,\tau_1)| \leq 2L, \; \forall\, \tau_1 \leq t \leq s \right\}.
$$
Observe that from the definition above and \eqref{osc_space}, we have
$$
h^- \leq v \leq h^+ \quad \mbox{on } \; \partial_p G_{\tau_1,t^\star}.
$$
To check this, we notice that in $B_1$, we have
\begin{equation}
h^-(x,\tau_1) \leq |v(0,\tau_1)-v(x,\tau_1)|-L+v(x,\tau_1) \leq v(x,\tau_1).
\end{equation}
For $(x,t) \in \partial B_1 \times (\tau_1,t^\star)$, we select $K=2L$, thus obtaining 
\begin{equation}\nonumber
\begin{array}{rcl}
h^-(x,t) & \leq & |v(0,\tau_1)-v(0,t)|+|v(0,t)-v(x,t)|-L-K + v(x,t)\\[0.2cm]
 & \leq & 2L-K + v(x,t) \\[0.2cm]
 & = & v(x,t).
\end{array}
\end{equation}
Under these choices, we also show that $h^+ \geq v$ on $\partial_p G_{\tau_1,t^\star}$. Consequently, from Proposition \ref{comparison}, we get
$$
h^- \leq v \leq h^+ \quad \mbox{in } \; G_{\tau_1,t^\star}.
$$
In particular, 
$$
-L-\overline{K}(t^\star-\tau_1)+v(0,\tau_1) \leq v(0,t^\star) \leq v(0,\tau_1)+L+\overline{K}(t^\star-\tau_1) 
$$
and so, taking into account that $t^\star-\tau_1 \leq \tau_2-\tau_1 \leq \kappa_0$, we conclude 
$$
|v(0,t^\star)-v(0,\tau_1)| \leq L+ \overline{K}(t^\star-\tau_1) \leq L+ \overline{K} \kappa_0.
$$
Since  
$$
\kappa_0 \leq \dfrac{L}{2(M+4n\Lambda L+16c_1 L^2)},
$$
we get 
$$|v(0,t^\star)-v(0,\tau_1)|\leq 3L/2<2L.$$
If $t^\star<\tau_2$, the estimate above contradicts the maximality of $t^\star$. Hence, we conclude that $t^\star=\tau_2$, as claimed.

Now, we argue as follows. Fix $t \in (-\kappa_0,0]$ arbitrary. If $G_{t,0} \subseteq \{v>1\}$, we have, by the previous analysis, 
$$|v(0,0)-v(0,t)| \leq 2L.$$
If $G_{t,0} \not\subseteq \{v >1\}$, then we can find (see Fig. \ref{fig1})
$$(x_i,t_i) \in \overline{G_{t,0}}, \quad i=1,2 ,$$
with $t_1 \leq t_2$, such that $v (x_i,t_i) \leq 1$ and
$$ \overline{B_1} \times (t,t_1) \cup \overline{B_1} \times (t_2,0) \subseteq \{v>1\}.$$

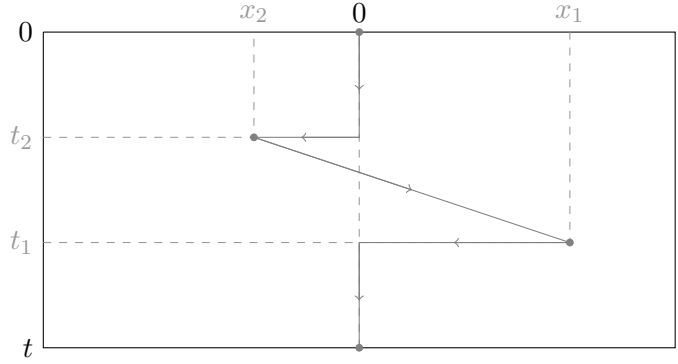
\begin{figure}[h!]
\centering
\begin{tikzpicture}[scale=1.4]
\draw (-3,3)node[left] {$0$} -- (-3,0) node[left] {$t$} -- (3,0) -- (3,3) -- (-3,3);
\draw[gray!80, dashed] (-3,2) node[left] {$t_2$} -- (-1,2) -- (-1,3) node[above] {$x_2$};
\draw[gray!80, dashed] (-3,1) node[left] {$t_1$} -- (2,1) -- (2,3) node[above] {$x_1$};
\draw[gray!80, dashed] (0,0) -- (0,3);
\filldraw [gray] (0,3) circle (0.9pt) node[above, black] {$0$};
\filldraw [gray] (-1,2) circle (0.9pt);
\filldraw [gray] (2,1) circle (0.9pt);
\filldraw [gray] (0,0) circle (0.9pt);
\draw[gray] (0,3) -- (0,2) -- (-1,2) -- (2,1) -- (0,1) -- (0,0);
\draw[->,gray] (0,3) -- (0,2.45);
\draw[->,gray] (0,2) -- (-0.55,2);
\draw[->,gray] (-1,2) -- (0.5,1.5);
\draw[->,gray] (2,1) -- (0.9,1);
\draw[->,gray] (0,1) -- (0,0.45);
\end{tikzpicture}
\vspace*{-0.4cm}
\caption{Oscillation control in time.}
\label{fig1}
\end{figure}

\bigskip

\noindent Then, from \eqref{osc_space} and \eqref{claim}, we have
\begin{eqnarray*}
|v(0,0)-v(0,t)| & \leq & |v(0,0)-v(0,t_2)|+|v(0,t_2)-v(x_2,t_2)|\\
& & +|v(x_2,t_2) - v(x_1,t_1)| + |v(x_1,t_1) - v(0,t_1)|\\
& & +|v(0,t_1)-v(0,t)|\\
& \leq & 2L + L + 2 + L + 2L \\
& \leq & 8L.
\end{eqnarray*}

\end{proof}

Note that, in the region $\{v>1\}$, viscosity solutions of equation \eqref{general_eq} solve \eqref{v>1},
for parameters $c_1=\delta \Lambda$, $c_2=\delta \lambda\|v\|_\infty^{-1}$ and $M=\|f\|_\infty$. Considering this and Theorem \ref{thm_holder}, we apply Lemma \ref{space_time} to obtain the next result.

\begin{theorem}\label{mu_regularity_time}
Let $v$ be a positive bounded viscosity solution of \eqref{general_eq} in $G_1$. 
For every $\mu \in (0,1)$, there exists a constant $C^\ast$, depending only on $\mu$, $\delta$, $\lambda$, $\Lambda$, $n$, $\|f\|_\infty$ and $\|v\|_\infty$, such that
\begin{equation}\nonumber
	\sup\limits_{t,s \, \in (-\kappa_0/2,0]} \frac{|v(x,t)-v(x,s)|}{|t-s|^\frac{\mu}{2}} \leq C^\ast,
\end{equation}
for  any $ \, x \in B_{1/2}$, where 
$$\kappa_0:= \dfrac{1}{4}\min\left\{1,\dfrac{8C}{\|f\|_\infty +16n\Lambda C+64\delta \Lambda C^2} \right\}$$
and $C$ is from Theorem \ref{thm_holder}.
\end{theorem}

\begin{proof}
Given $0<\mu<1$ and $0<\kappa \leq 1/2$, we define, in $G_{1}$,
$$
w(x,t):= \frac{v\left(\kappa x, \kappa^2 t\right)}{\kappa^\mu}.
$$
Applying Theorem \ref{thm_holder} with $y=0$, we obtain
$$|v(x,t)-v(0,t)| \leq C |x|^\mu, \quad \forall \, x \in B_{1/2}, \quad \forall \,  t \in (-1/4, 0]. $$
Thus, since $0<\kappa \leq 1/2$,
$$\sup\limits_{x \in B_\kappa} |v(x,\kappa^2 t)-v(0,\kappa^2 t)| \leq C \kappa^\mu, \quad \forall \,  t \in (-1/4, 0]$$
and we conclude
$$
\sup\limits_{x \in B_1}|w(x,t)-w(0,t)| \leq C,
$$
uniformly in $-1/4 < t \leq 0$.

Now, observe that in the region $\{w>1\} \cap G_{1/2}$, we have
\begin{equation}\nonumber
\begin{array}{ccc}
\mathcal{M}_{\lambda, \Lambda}^+(D^2 w) + b_1
|\nabla w|^2 - \partial_t w & \geq & -M, \\[0.2cm]
\mathcal{M}_{\lambda, \Lambda}^-(D^2 w) + b_2
|\nabla w|^2 - \partial_t w & \leq & M,
\end{array}
\end{equation}
where $b_i=\kappa^\mu c_i \leq c_i$, for $i=1,2$.  Therefore, from Lemma \ref{space_time}, we obtain 
$$
\sup\limits_{t \in (-\kappa_0,0]}|w(0,t)-w(0,0)| \leq 8C,
$$
which implies
$$
\sup\limits_{s \in (-r^2,0]}|v(0,s)-v(0,0)| \leq \frac{8C}{\sqrt{\kappa_0^\mu}} \, r^\mu,
$$
for each $0 \leq r \leq \frac{\sqrt{\kappa_0}}{2}$.

By a standard scaling and translation procedure, we obtain the result with
$$C^\ast =  8C \left( \frac{2}{\sqrt{\kappa_0}} \right)^\mu.$$

\end{proof}

\section{Improved regularity estimates}\label{loc_reg_sec}

In this section, our primary goal is to obtain pointwise oscillation estimates in space and time for positive viscosity solutions of \eqref{Ee}. We drop the $\varepsilon$ in $u_\varepsilon$  for simplicity of notation. The core of the argument is to derive  regularity estimates for  
\begin{equation}\label{ugamma}
	u_\gamma(x,t):= u(x,t)^{\frac{2-\gamma}{2}},
\end{equation}
which are obtained by accessing PDE information for $u_\gamma$, using the fact that $u$ solves \eqref{Ee}. 

Denoting $v:=u_\gamma$, and recalling \eqref{alpha}, we compute
\begin{equation}\label{Dv}
	\partial_t v = \frac{1}{1+\alpha}u^{-\frac{\gamma}{2}}\, \partial_t u, \quad \quad \nabla v = \frac{1}{1+\alpha}u^{-\frac{\gamma}{2}}\nabla u
\end{equation}
and
\begin{equation}\label{D2v}
	D^2v = \frac{1}{1+\alpha}\left( \frac{1}{1+\alpha}-1\right)u^{-\frac{\gamma}{2}-1}\nabla u \otimes \nabla u + \frac{1}{1+\alpha}u^{-\frac{\gamma}{2}} D^2u.
\end{equation}
Now, plugging \eqref{Dv} into \eqref{D2v} yields
\begin{equation}\label{D2v2}
	D^2v + \alpha\, v^{-1} \nabla v \otimes \nabla v = \frac{1}{1+\alpha}v^{-1} \left(u^{1-\gamma} D^2u \right).  
\end{equation}
From the PDE satisfied by $u$ and assumption {\bf (A3)}, we derive 
\begin{equation}\nonumber
	\begin{array}{rcl}
		F \left(x,t,D^2v + \alpha\,v^{-1} \nabla v \otimes \nabla v \right) & = & \displaystyle \frac{1}{1+\alpha}v^{-1} \left(u^{1-\gamma} F(x,t,D^2u) \right) \\[0.4cm]
		& = & \displaystyle \frac{1}{1+\alpha}v^{-1}(\mathcal{B}_\varepsilon(u)+u^{1-\gamma} \partial_t u) \\[0.4cm]
		& = & \displaystyle \frac{1}{1+\alpha}v^{-1}\mathcal{B}_\varepsilon(v^{1+\alpha})+\partial_t v.
	\end{array} 
\end{equation}
Therefore, we conclude that $v$ satisfies \eqref{general_eq}, that is,
$$F\left(x,t, D^2 v + \delta\,v^{-1} \nabla v \otimes \nabla v \right) - \partial_t v=f(x,t)\,v^{-1} \quad \mbox{in } \; \Omega_T,$$
in the viscosity sense, for some bounded function $f$, with $\|f\|_\infty \leq 1$, and $\delta = \alpha$. Although the computations above have been conducted formally, it is standard to justify that $v$ satisfies the PDE in the viscosity sense. 

Considering these facts, we obtain the following result from Theorems \ref{thm_holder} and \ref{mu_regularity_time}.

\begin{corollary}\label{undou}
Let $u$ be a positive bounded viscosity solution of \eqref{Ee}. For every $\mu \in (0,1)$, there exists a universal constant $C>0$, depending on $\mu$, $\gamma$, $\lambda$, $\Lambda$, $n$ and  $\|u\|_{L^\infty(G_1)}$, but not depending on $\varepsilon$, such that
\begin{equation}
\sup\limits_{(x,t) \in G_r}u(x,t) \leq \left( Cr^\mu+u(0,0)^{\frac{1}{1+\alpha}}  \right)^{1+\alpha}, 
\label{asym_growth}
\end{equation}
for each $0< r \leq \frac{\sqrt{\kappa_0}}{2}$, with $\kappa_0$ as in Theorem \ref{mu_regularity_time}.\end{corollary}

\begin{proof}
First, from Theorem \ref{thm_holder}, we obtain 
$$\sup\limits_{x\in B_r}  \left| u(x,t)^{\frac{1}{1+\alpha}} - u(0,t)^{\frac{1}{1+\alpha}} \right| \leq C^\prime r^\mu $$
$$\Longrightarrow \sup\limits_{x\in B_r} u(x,t)^{\frac{1}{1+\alpha}}  \leq C^\prime r^\mu + u(0,t)^{\frac{1}{1+\alpha}},\qquad $$
for every $r\leq 1/2$ and $t \in (-1/4, 0]$.

Then, from Theorem \ref{mu_regularity_time},
$$\sup\limits_{t \in (-r^2,0]}  \left| u(0,t)^{\frac{1}{1+\alpha}} - u(0,0)^{\frac{1}{1+\alpha}} \right| \leq C^{\prime \prime} r^\mu$$
$$\Longrightarrow\sup\limits_{t \in (-r^2,0]}  u(0,t)^{\frac{1}{1+\alpha}}  \leq C^{\prime \prime}r^\mu+u(0,0)^{\frac{1}{1+\alpha}},\qquad $$
for every $r\leq \frac{\sqrt{\kappa_0}}{2}$.

Combining the two, we obtain the result.

\end{proof}

We conclude this section by deriving gradient estimates for positive viscosity solutions of \eqref{Ee}. First, we provide uniform local $C^{Lip,\frac{1}{2}}-$regularity estimates.

\begin{theorem}\label{Holderreg}
Let $u$ be a positive bounded viscosity solution of \eqref{Ee}.	There exists a universal constant $C>0$, depending on $\gamma$, $\lambda$, $\Lambda$, $n$ and  $\|u\|_{L^\infty(G_1)}$, but not depending on $\varepsilon$, such that
\begin{equation}\label{lipschitzest}
\sup\limits_{(x,t) \in G_{r}} |u(x,t)-u(0,0)| \leq Cr,
\end{equation}
for each $0< r \leq \frac{\sqrt{\kappa_0}}{2}$.
\end{theorem} 

\begin{proof}
Fix $0< r \leq \frac{\sqrt{\kappa_0}}{2}$. Assuming $r \geq u(0,0)$, we apply estimate \eqref{asym_growth} for $\mu=1/(1+\alpha)<1$, to obtain
\begin{equation}
\sup\limits_{(x,t) \in G_r}u(x,t) \leq \left( Cr^{\frac{1}{1+\alpha}} +u(0,0)^{\frac{1}{1+\alpha}}  \right)^{1+\alpha} \leq \left(C+1  \right)^{1+\alpha}r,
\end{equation}
and \eqref{lipschitzest} follows from the triangular inequality.

In the case $0<r < u(0,0)$, we consider
$$
w(x,s):= \frac{u(\kappa x, \kappa^2 s)}{\kappa}, \qquad (x,t) \in G_1,
$$
for $\kappa:= u(0,0)$. Note that, up to a universal rescaling, we can assume $\kappa \leq 1/4$. In addition, from Remark \ref{rem}, we find that $w$ is a positive viscosity solution  of 
$$F_\kappa(x,t,D^2 w)  -\partial_s w = \kappa^{\gamma}\mathcal{B}_{\epsilon}(w)\,w^{\gamma-1} $$
in $G_1$, where $\epsilon =\varepsilon \kappa^{-\frac{1}{1+\alpha}}$ and
$$F_\kappa(x,t,M):=\kappa F(\kappa x,\kappa^2 t,\kappa^{-1}M). $$
Again from estimate \eqref{asym_growth}, with $\mu=1/ (1+\alpha)$, we have
$$
w(x,s) = \frac{u(\kappa x, \kappa^2 s)}{\kappa}\leq  \frac{\left(C \kappa^{\frac{1}{1+\alpha}}+u(0,0)^\frac{1}{1+\alpha}\right)^{1+\alpha}}{\kappa} = (C+1)^{1+\alpha} \; \mbox{ in }\; G_1.
$$
We also observe that $w(0,0)=w_\gamma(0,0)=1$ (cf. \eqref{ugamma}). Hence, by continuity in space and time for $\omega_\gamma$, there exists a universal parameter $\delta_0>0$, such that
$$
w(x,s) = w_\gamma(x,s)^{{1+\alpha}} \geq \left( \frac{1}{2}\right)^{{1+\alpha}},  \quad \forall (x,s) \in  G_{\delta_0}.
$$    
This implies that $w$ solves a $(\lambda,\Lambda)$-parabolic equation with bounded source term $f$,
$$
F_\kappa(x,t,D^2 w)  -\partial_t w = f \quad \mbox{ in }\; G_{\delta_0}.
$$
By classical regularity estimates, see for example \cite{W2}, it follows that
	\begin{equation}\nonumber
	 \sup\limits_{(x,t) \in G_r} |w(x,t)-w(0,0)| \leq C r,
	\end{equation}
for $0< r \leq \delta_0/2$, and so
	\begin{equation}\nonumber
	\sup\limits_{(x,t) \in G_r} |u(x,t)-u(0,0)| \leq C r,
	\end{equation}
for $0< r \leq \kappa\delta_0/2$. Finally, we consider radii $\kappa\delta_0/2 < r < \kappa$. Using estimate \eqref{asym_growth}, once again with $\mu=1/ (1+\alpha)$, we conclude
\begin{eqnarray*}
\sup\limits_{(x,t) \in G_r} |u(x,t)-u(0,0)| & \leq & \sup\limits_{(x,t) \in G_\kappa} u(x,t)+u(0,0) \\
&  \leq & \left\{(C+1)^{1+\alpha} + 1 \right\} \kappa \\
& \leq & \frac{2}{\delta_0}\left\{(C+1)^{1+\alpha} + 1 \right\}  r.
\end{eqnarray*}
Therefore, estimate \eqref{lipschitzest} holds for each $0< r \leq \frac{\sqrt{\kappa_0}}{2}$.

\end{proof}

As a consequence of Theorem \ref{Holderreg}, we obtain the following gradient bound, which shall be used crucially to establish this paper's main result. Note that the estimate therein is invariant under scaling for functions defined in Remark \ref{rem}. 
  
\begin{corollary}\label{riviera}
Let $u$ be a positive bounded viscosity solution of \eqref{Ee}. For every $0<\theta < \gamma$, there exists a constant $C$, depending only on $\theta$, $\gamma$, $\lambda$, $\Lambda$, $n$ and  $\|u\|_{L^\infty(G_1)}$, but not depending on $\varepsilon$, such that
\begin{eqnarray}\label{gradestim}
		|\nabla u(x,t)|^2 \leq Cu(x,t)^{\theta} \quad {\rm in} \ G_{1/2}.
\end{eqnarray}

\end{corollary}

\begin{proof}
Fix the parameter $0<\theta < \gamma$ and let $\beta := \frac{\theta}{2-\theta} < \alpha$. Consider $(x_0,t_0)\in G_{1/2}$ and 
$$
M \geq \|u\|_\infty\kappa_0^{-\frac{1+\beta}{2}}
\quad \mbox{such that} \quad
r_0:=\left(\frac{u(x_0,t_0)}{M}\right)^\frac{1}{1+\beta} \leq \sqrt{\kappa_0}.
$$
Now, we define in $G_1$ the rescaled function
	$$
	w(x,t)=\frac{u(x_0 + r_0x, t_0 + r_0^2t)}{r_0^{1+\beta}}.
	$$
By estimate \eqref{asym_growth}, choosing $\mu= \frac{1+\beta}{1+\alpha} < 1$, we have
\begin{eqnarray*}
	\sup_{G_1}w & = & \sup_{G_1}\frac{u(x_0 + r_0x, t_0 + r_0^2t)}{r_0^{1+\beta}} \\
	& \leq & \frac{\left( C r_0^{\frac{1+\beta}{1+\alpha}} + u(x_0,t_0)^{\frac{1}{1+\alpha}} \right)^{1+\alpha}}{r_0^{1+\beta}}\\
	& \leq & \left( C + M^{\frac{1}{1+\alpha}}  \right)^{1+\alpha}.
\end{eqnarray*}

Moreover, using Remark \ref{rem}, $w$ is a positive viscosity solution of
$$F_0(x,t,D^2w)-\partial_tw= r_0^{(1+\beta) (\gamma -1) + 1-\beta}\B_{\epsilon}(w)w^{\gamma-1} $$
in $G_1$, where $\epsilon =\varepsilon r_0^{-\mu}$ and 
$$F_0(x,t,M):=r_0^{1-\beta}F \left(x_0+r_0x,t_0+r_0^2t, r_0^{\beta-1}M \right).$$
Note that 
$$
(1+\beta) (\gamma -1) + 1-\beta >0,
$$ 
since $\theta < \gamma$. By Theorem \ref{Holderreg}, we have
$$
|\nabla u(x_0,t_0)| = r_0^{\beta} | \nabla w(0,0)|\leq Cr_0^{\beta} = C  \left(\frac{u(x_0,t_0)}{M}\right)^{\frac{\beta}{1+\beta}} = \overline{C} u(x_0,t_0)^{\frac{\theta}{2}},
$$
for a universal constant $\overline{C}>0$.

\end{proof}

\section{Sharp local $C^{1+\beta, \frac{1+\beta}{2}}-$regularity estimates}\label{optsection} 

In this section, we derive asymptotically optimal regularity estimates for viscosity solutions of \eqref{Ee}, which are uniform in $\varepsilon$. Let $0< \alpha_F <1$ be an exponent such that viscosity solutions of 
$$F(x,t,D^2 h)=\partial_t h \quad \mbox{in} \quad G_1$$
are of class  $C^{1+\alpha_F, \frac{1+\alpha_F}{2}}$, \textit{i.e.}, they satisfy
$$\sup\limits_{(x,t) \in G_\rho (y,s)} \left| h(x,t)-h(y,s)-\nabla h(y,s)\cdot (x-y) \right| \leq C\|h\|_\infty \rho^{1+\alpha_F},$$
for any $(y,s) \in G_{1/2}$ and $0< \rho \leq 1/4$. 
Recall $\alpha:=\frac{\gamma}{2-\gamma}.$

\begin{theorem}\label{localregularity}
Every positive bounded viscosity solution $u$ of \eqref{Ee} is of class $C^{1+\beta, \frac{1+\beta}{2}}$ in $G_{1/2}$, for any $\beta < \min\{\alpha,\alpha_F\}$. Furthermore, there exists a constant $C$, depending only on $\gamma$, $\lambda$, $\Lambda$, $n$ and  $\|u\|_{L^\infty(G_1)}$, but not depending on $\varepsilon$, such that \begin{equation}\label{mainestimate}
	\sup\limits_{(x,t) \in G_\rho(y,s)} |u(x,t)-u(y,s)-\nabla u(y,s)\cdot (x-y)| \leq C \rho^{1+\beta},
\end{equation}
for any $(y,s) \in G_{1/2}$ and $0< \rho \leq \frac{\sqrt{\kappa_0}}{2}$.
\end{theorem}

\begin{proof}
Let $(y,s) \in G_{1/2}$ and define
\begin{equation} \label{adriatico}
\rho_\star := \left[ \frac{u(y,s)}{M} \right]^{\frac{1}{1+\beta}},
\end{equation}
for $M>1$ large enough such that $\rho_\star \leq \frac{\sqrt{\kappa_0}}{2}$.

We split the proof into two cases, considering first $\rho_\star \leq \rho \leq \frac{\sqrt{\kappa_0}}{2}$ (see Fig. \ref{fig2}). We use Corollary \ref{undou}, with $\mu=\frac{1+\beta}{1+\alpha}<1$, and Corollary \ref{riviera}, with $\theta = \frac{2\beta}{1+\beta}<\gamma$, to derive
\begin{equation}\nonumber
\begin{array}{rl}
 & \sup\limits_{(x,t) \in G_\rho(y,s)} |u(x,t)-u(y,s)-\nabla u(y,s)\cdot (x-y)| \\[0.4cm] 
 \leq & \sup\limits_{(x,t) \in G_\rho(y,s)} u(x,t) + u(y,s) +|\nabla u(y,s)| \rho \\[0.4cm]
 \leq & \left( C \rho^{\frac{1+\beta}{1+\alpha}}+u(y,s)^{\frac{1}{1+\alpha}} \right)^{1+\alpha} + u(y,s) + C u(y,s)^{\frac{\beta}{1+\beta}} \rho \\[0.4cm]
 \leq & \left( C \rho^{\frac{1+\beta}{1+\alpha}}+(M{\rho}^{1+\beta})^{\frac{1}{1+\alpha}} \right)^{1+\alpha} + M{\rho}^{1+\beta} + C (M{\rho}^{1+\beta})^{\frac{\beta}{1+\beta}} \rho \\[0.4cm]
 \leq & \left[\left( C +M^{\frac{1}{1+\alpha}} \right)^{1+\alpha} + M + C M^{\frac{\beta}{1+\beta}} \right] {\rho}^{1+\beta}.
\end{array}
\end{equation}

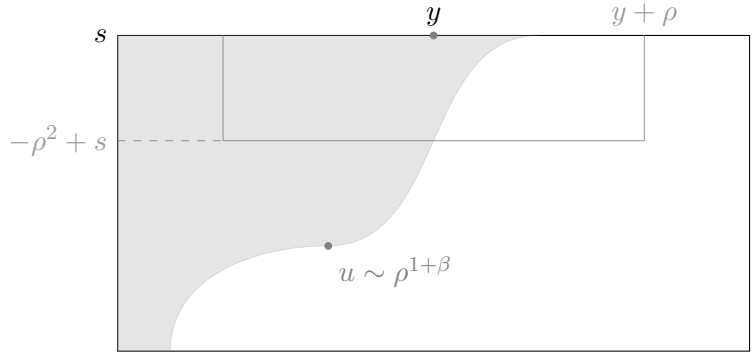
\begin{figure}[h!]
\centering
\hspace*{-1.5cm}
\begin{tikzpicture}[scale=1.4]
\draw (-3,3)node[left] {$s$} -- (-3,0) -- (3,0) -- (3,3) -- (-3,3);
\draw[gray!80] (-2,3) -- (-2,2) -- (2,2) -- (2,3) node[above] {$y+\rho$};
\draw[gray!80, dashed] (-3,2) node[left] {$-\rho^2+s$} -- (-2,2);
\filldraw [gray] (0,3) circle (0.9pt) node[above, black] {$y$};
\filldraw [gray] (-1,1) circle (0.9pt) node[below right] {$u \sim \rho^{1+\beta}$};
\filldraw [gray, opacity=0.2] (-3,3) -- (-3,0) -- (-2.5,0) to [out=90,in=180] (-1,1)  to [out=0,in=180] (1,3) -- (-3,3);
\end{tikzpicture}
\caption{Oscillation estimates in $G_{\rho}(y,s)$ for the singular case $u(y,s) \lesssim \rho^{1+\beta}$.}
\label{fig2}
\end{figure}

\bigskip

Next, we consider the case $0< \rho < \rho_\star$. The function $w$ defined by
$$w(x,t):= \frac{u(y+\rho_\star x, s+\rho_\star^2 t)}{\rho_\star^{1+\beta}}$$
is a positive viscosity solution in $G_1$ of 
$$
F_\kappa(x,t,D^2 w)  -\partial_s w = \rho_\star^{(1+\beta)(\gamma -1)+1-\beta}\mathcal{B}_{\epsilon}(w)\,w^{\gamma-1} ,$$
where $\epsilon =\varepsilon \rho_\star^{-\frac{1+\beta}{1+\alpha}}$.
From estimate \eqref{asym_growth}, with $\mu=\frac{1+\beta}{1+\alpha}$, we have that
\begin{eqnarray*}
w(x,t) & =  & \frac{u(y+\rho_\star x, s+\rho_\star^2 t)}{\rho_\star^{1+\beta}}\\
& \leq & \frac{\left(C \rho_\star^{\frac{1+\beta}{1+\alpha}}+u(y,s)^\frac{1}{1+\alpha}\right)^{1+\alpha}}{\rho_\star^{1+\beta}} \\
& = & (C+M^\frac{1}{1+\alpha})^{1+\alpha}\\
& \leq & 2^\alpha \left( C^{1+\alpha} +M \right),
\end{eqnarray*}
for $(x,t) \in G_1$. In addition, we observe that $w(0,0)=M>1$. Hence, from Proposition \ref{Holderreg},  there exists a universal parameter $\delta_0$, such that
$$
w(x,s)  \geq \frac{1}{2},  \quad \mbox{for each } \; (x,s) \in  G_{\delta_0}.
$$    
This implies that $w$ solves a uniformly parabolic equation with bounded source term $f$, namely
$$
F_\kappa(x,s,D^2 w)  -\partial_t w = f(x,s) \quad \mbox{ in }\; G_{\delta_0}.
$$
By classical regularity estimates (see for example \cite{W2} and \cite{daST}) and since $\beta < \alpha_F$, we have
\begin{equation}\nonumber
	\sup\limits_{(x,t) \in G_r} \left| w(x,t)-w(0,0)-\nabla w(0,0)\cdot x) \right| \leq C r^{1+\beta},
\end{equation}
for $0< r \leq \delta_0/2$, and so
\begin{equation}\nonumber
	\sup\limits_{(x,t) \in G_r(y,s)} |u(x,t)-u(y,s)-\nabla u(y,s)\cdot (x-y)| \leq C r^{1+\beta},
\end{equation}
for $0< r \leq \rho_\star\delta_0/2$. Finally, for radii $\rho_\star\delta_0/2 < r < \rho_\star$, we conclude
\begin{equation}\nonumber
	\begin{array}{rl}
 & \sup\limits_{(x,t) \in G_r(y,s)} |u(x,t)-u(y,s)-\nabla u(y,s)\cdot (x-y)| \\[0.5cm] 
 \leq & \sup\limits_{(x,t) \in G_{\rho_\star}(y,s)} |u(x,t)-u(y,s)-\nabla u(y,s)\cdot (x-y)| \\[0.5cm]
 \leq & C \rho_\star^{1+\beta} \\[0.2cm]
 \leq & \overline{C} \, r^{1+\beta},
\end{array}
\end{equation}
where the second inequality follows from the first case.
Therefore, estimate \eqref{mainestimate} holds for every $0<r < \rho_\star$, which completes the proof.

\end{proof}

\section{The limiting free boundary problem}\label{limitingsec}

In this final section, we pass to the limit in \eqref{Ee} as $\varepsilon \searrow 0$. We will show that viscosity solutions $u_\varepsilon$ of \eqref{Ee} converge to a function $u_0$ solving \eqref{E0} in the viscosity sense. Then we obtain regularity estimates for $u_0$ from the regularity estimates obtained for viscosity solutions of the penalised equation \eqref{Ee}, thus getting sharp regularity results for limiting solutions of \eqref{E0}. 

We start with the compactness of the sequence $\left( u_\varepsilon \right)_\varepsilon$. From Remark \ref{ABP}, we know the sequence is uniformly bounded, and from Theorem \ref{Holderreg}, it is equicontinuous. Using Ascoli--Arzel\`a Theorem, we conclude there exists a continuous function $u_0$ such that, up to a subsequence 
$$u_\varepsilon  \longrightarrow u_0,$$
uniformly on compacts.

\begin{theorem}\label{bara}
The limiting function $u_0$ is a nonnegative bounded viscosity solution of \eqref{E0}.  At a free boundary point $(y,s)$, $u_0$ is of class $C^{1+\beta,\frac{1+\beta}{2}}$, for every
$$\beta < \frac{\gamma}{2-\gamma},$$
and
\begin{equation}\label{hulk}
	\sup\limits_{(x,t) \in G_\rho(y,s)} u(x,t) \leq C \rho^{1+\beta},
\end{equation}
for any $G_\rho(y,s) \Subset \Omega_T$. Moreover,  $u_0 $ is locally of class $C^{1+\beta, \frac{1+\beta}{2}}$, for every
$$\beta < \min \left\{ \frac{\gamma}{2-\gamma},\alpha_F \right\},$$
and 
\begin{equation}\label{ioca}
	\sup\limits_{(x,t) \in G_\rho(y,s)} |u_0(x,t)-u_0(y,s)-\nabla u_0(y,s)\cdot (x-y)| \leq C \rho^{1+\beta},
\end{equation}
for any $G_\rho(y,s) \Subset \Omega_T$. The constant $C$ depends only on $\gamma$, $\lambda$, $\Lambda$, $n$ and  $\|u_0\|_{\infty}$.
\end{theorem}

\begin{proof}
Fix a point $(x_0,t_0) \in \{u>0\} \cap \Omega_T$. Let $u_0(x_0,t_0)= c_0 > 0$ and, by continuity, select $\rho_0 >0$ such that 
$$u_0(x,t) > \frac{c_0}{2} \quad \mbox{in } G_{\rho_0} (x_0,t_0).$$
By the uniform convergence $u_\varepsilon  \longrightarrow u_0$ on compact sets, for $\varepsilon \ll 1$, we have
$$u_\varepsilon (x,t) > \frac{c_0}{4} > (1+\sigma_0) \varepsilon^{1+\alpha}$$ 
and thus $u_\varepsilon$ solves 
$$F(x,t,D^2u_\varepsilon)  -\partial_t u_\varepsilon = \gamma u_\varepsilon^{\gamma-1} \quad \mbox{in } G_{\rho_0} (x_0,t_0)$$
in the viscosity sense. By the stability of viscosity solutions under uniform limits (cf. \cite{IS}), we conclude $u_0$ is indeed a viscosity solution of \eqref{E0}.

The regularity at a free boundary point and estimate \eqref{hulk} follow from passing to the limit in \eqref{asym_growth}. 

That $u_0 \in C_{loc}^{1+\beta, \frac{1+\beta}{2}}$, for any $\beta < \min \left\{ \frac{\gamma}{2-\gamma},\alpha_F \right\}$, and the validity of estimate \eqref{ioca} are direct consequences of the uniform in $\varepsilon$ nature of the estimate in Theorem \ref{localregularity}.

\end{proof}

\medskip

{\small \noindent{\bf Acknowledgments.} DJA thanks the Abdus Salam International Centre for Theoretical Physics (ICTP) for the hospitality during his research visits. DJA is partially supported by Conselho Nacional de Desenvolvimento Científico e Tecnológico (CNPq) grant 311138/2019-5 and Paraíba State Research Foundation (FAPESQ) grant 2019/0014.  GSS is supported by a Capes PhD scholarship. JMU is partially supported by the King Abdullah University of Science and Technology (KAUST) and by the Centre for Mathematics of the University of Coimbra (UIDB/00324/2020, funded by the Portuguese Government through FCT/MCTES).}

\medskip

\bibliographystyle{amsplain, amsalpha}

\end{document}